\documentclass[12pt]{amsart}

\usepackage{amsthm}
\usepackage{amsmath}
\usepackage{mathrsfs}
\usepackage{amssymb}
\usepackage[dvipsnames]{xcolor}
\usepackage{tikz-cd}
\usepackage{enumitem}
\usepackage{hyperref}
\hypersetup{colorlinks = true,
	linkcolor = MidnightBlue,
	urlcolor = BrickRed,
	citecolor = MidnightBlue}

\title[Pro-representability of $K^M$-cohomology in weight 3]{Pro-representability of $K^M$-cohomology in weight 3 generalizing a result of Bloch}
\author{Eoin Mackall}
\email{eoinmackall \emph{at} gmail.com}
\urladdr{\url{www.eoinmackall.com}}
\date{\today}

\newtheorem{thm}{Theorem}[section]

\newtheorem{lem}[thm]{Lemma}

\theoremstyle{definition}
\newtheorem{defn}[thm]{Definition}
\newtheorem{exmp}[thm]{Example}
\newtheorem{rmk}[thm]{Remark}

\newcounter{item}

\newcommand{\HH}{\mathrm{H}}

\newcommand{\Spec}{\mathrm{Spec}}
\newcommand{\R}{\mathcal{R}}

\begin{document}
\maketitle
\begin{abstract}
We generalize a result, on the pro-representability of Milnor $K$-cohomology groups at the identity, that's due to Bloch. In particular, we prove, for $X$ a smooth, proper, and geometrically connected variety defined over an algebraic field extension $k/\mathbb{Q}$, that the functor \[\mathscr{T}_{X}^{i,3}(A)=\ker\left(\HH^i(X_A,\mathcal{K}_{3,X_A}^M)\rightarrow \HH^i(X,\mathcal{K}_{3,X}^M)\right),\] defined on Artin local $k$-algebras $(A,\mathfrak{m}_A)$ with $A/\mathfrak{m}_A\cong k$, is pro-representable provided that certain Hodge numbers of $X$ vanish.
\end{abstract}

\section{Introduction}
Bloch \cite{MR371891} has shown that the Milnor $K$-cohomology $\HH^i(X,\mathcal{K}^M_{2,X})$, for a variety $X$ defined over a number field $k$, is an interesting object amenable to study by deformation theory. More precisely, Bloch found a necessary and sufficient criterion for determining pro-representability of the functor which sends a local Artinian $k$-algebra $A$ to the kernel of the canonical map \[\HH^i(X_A,\mathcal{K}^M_{2,X_A})\rightarrow \HH^i(X,\mathcal{K}^M_{2,X})\] in terms of the vanishing of some Hodge numbers of $X$. Recall \cite{MR217093} that a functor $F:\mathsf{Art}_k\rightarrow \mathsf{Set}$ from the category $\mathsf{Art}_k$ of local Artinian $k$-algebras $(A,\mathfrak{m}_A)$ with residue field $A/\mathfrak{m}_A\cong k$ is called \textit{pro-representable} if there is a complete local Noetherian $k$-algebra $(R,\mathfrak{m}_R)$ such that: \begin{enumerate}[nosep, leftmargin=*]\item $(R/\mathfrak{m}_R^n,\mathfrak{m}_R/\mathfrak{m}_R^n)$ is an object in $\mathsf{Art}_k$ for all integers $n\geq 1$, 
\item there exists a canonical isomorphism of functors $F\simeq h_R$ between $F$ and the functor $h_R:\mathsf{Art}_k\rightarrow \mathsf{Set}$ characterized by the assignment $h_R(A)=\mathrm{Hom}_{\text{local } k-\text{alg}}(R,A)$ for any object $(A,\mathfrak{m}_A)$ of $\mathsf{Art}_k$.\end{enumerate}

A key ingredient necessary for the proof of Bloch's criterion is the existence of a certain logarithmic comparison isomorphism between the kernel of the restriction map $K_2(B)\rightarrow K_2(A)$, for a local $\mathbb{Q}$-algebra $A$ and an $A$-algebra $B$ with $A$-augmentation $\rho:B\rightarrow A$ so that the kernel $J=\mathrm{ker}(\rho)$ is nilpotent, and a certain subquotient of the group of absolute K{\"a}hler differentials for $B$ (this was also proved in \cite{MR371891}).

Since \cite{MR371891}, there have been some attempts at generalizing Bloch's results in different directions. To name some of the successes: Stienstra \cite{MR718076} generalized the pro-representability criterion to surfaces defined over more general base fields; Maazen and Stienstra \cite{MR472795} provided an explicit presentation for $K_2(R)$ for a large class of rings $R$, allowing them to recover the relation to differentials provided by Bloch; and, more recently, Bloch's logarithmic comparison isomorphism has been extended to higher Milnor $K$-groups by both Gorchinskiy and Tyurin \cite{MR3859378}, and independently Dribus \cite{https://doi.org/10.48550/arxiv.1402.2222}.

In this note, we show how the theorems of Gorchinskiy and Tyurin, and Dribus, sheafify to allow study of pro-representability for functors that are related to the higher Milnor $K$-cohomology groups $\HH^i(X,\mathcal{K}^M_{n,X})$. Our main result is a sufficiency criterion for pro-representability in the case of weight 3 Milnor $K$-cohomology. That is to say, we prove that for a variety $X$ satisfying a host of assumptions (specifically, $X$ should be very nice geometrically, defined over an algebraic extension $k/\mathbb{Q}$, and have some vanishing Hodge numbers), the functor which assigns to an Artin local $k$-algebra $(A,\mathfrak{m}_A)$ with residue field $A/\mathfrak{m}_A\cong k$ the group $\HH^i(X_A,\mathcal{K}^M_{3,X_A})$ is pro-representable at the identity (Theorem \ref{thm: Bloch3}). Our proof is also descriptive, like Bloch's, in the sense that we prove pro-representability by constructing an isomorphism \[\ker\left(\HH^i(X_A,\mathcal{K}_{3,X_A}^M)\rightarrow \HH^i(X,\mathcal{K}_{3,X}^M)\right)\cong \HH^i(X,\Omega^2_{X/k})\otimes_k \mathfrak{m}_A\] which is of independent interest.

Lastly, we point out that most of the results proved here actually apply to all Milnor $K$-cohomology groups of arbitrary weight $n\geq 1$. Our main theorem is only limited to weight $3$ because of the author's inability to determine the vanishing of certain sheaf cohomology groups related to differential $p$-forms for $p>2$. So, if one could generalize the vanishing result of Lemma \ref{lem: vanc2}, then it should be possible to prove an appropriate sufficiency criterion for pro-representability in any weight.\\ 

\noindent\textbf{Acknowledgments}. The author would like to thank an anonymous referee whose comments and suggestions have greatly increased the readability of the given article.\\

\noindent\textbf{Conventions}. We use the following conventions throughout:
\begin{itemize}[leftmargin=*]
	\item a variety is an integral scheme that is separated and of finite type over a base field.
\end{itemize}
	
\noindent\textbf{Notation}. We use the following notation throughout:
\begin{itemize}[leftmargin=*]
	\item $k/\mathbb{Q}$ is a fixed algebraic extension
	\item $(A,\mathfrak{m}_A)$ is an Artin local $k$-algebra with structure map $s:k\rightarrow A$ and residue field $A/\mathfrak{m}_A=k$
	\item $X$ is a smooth $k$-variety with structure map $\phi:X\rightarrow \Spec(k)$ and $X_A=X\times_k \Spec(A)$ is the trivial deformation of $X$ over $\Spec(A)$
	\item $\pi_1:X_A\rightarrow X$ and $\pi_2:X_A\rightarrow \Spec(A)$ are the first and second projections respectively.
	\item $\mathcal{J}\subset \mathcal{O}_{X_A}$ is the ideal sheaf of the reduction $\rho:X\rightarrow X_A$.
\end{itemize}

\section{Generalizing Bloch's result}
The purpose of this section is to collect preliminaries (mostly without proof) that directly generalize known results from the literature.
\begin{defn}
We write $\R^1_\rho$ for the kernel sheaf \[\R^1_\rho=\ker\left(\Omega^1_{X_A/k}\twoheadrightarrow \rho_* \Omega_{X/k}^1\right).\]
This sheaf is sometimes denoted $\Omega^1_{X_A,\mathcal{J}}$ or $\Omega^1_{X\times A, X\times \mathfrak{m}_A}$ in the literature. Note that $\R^1_\rho$ is coherent. 
\end{defn}

\begin{defn}
For any $n>1$ we define $\R^n_\rho$ as the kernel \[\R^n_\rho=\ker\left(\Omega^n_{X_A/k}\twoheadrightarrow \wedge^n\left(\rho_* \Omega_{X/k}^1\right)\right).\] This is equal to the image sheaf \[\R^n_\rho=\mathrm{Im}\left( \R^1_\rho \otimes_{\mathcal{O}_{X_A}} \Omega^{n-1}_{X_A/k}\twoheadrightarrow \Omega^n_{X_A/k}\right).\]
Note that $\R^n_\rho$ is also coherent.
\end{defn}

\begin{rmk}
The differentials $d^n:\Omega^n_{X_A/k}\rightarrow \Omega_{X_A/k}^{n+1}$ restricted to the subsheaves $\R^n_\rho$ fit together in a complex \[\R^0_\rho:=\mathcal{J}\rightarrow \R^1_\rho \xrightarrow{d^1} \R^2_\rho \xrightarrow{d^2} \cdots \xrightarrow{d^{i-1}} \R^i_\rho \xrightarrow{d^i} \cdots.\] We will often write $d$ to mean any one of these differentials.
\end{rmk}
	
For any scheme $Y$ we write $\mathcal{K}^M_{n,Y}$ for the Zariski sheaf associated to the presheaf of Milnor $K$-groups defined by the assignment \[U \rightsquigarrow K^M_n(\mathcal{O}_Y(U))\] for any open $U\subset Y$. There is then an exact sequence \[ 0\rightarrow \mathcal{K}^M_{n,\rho}\rightarrow \mathcal{K}^M_{n,X_A}\rightarrow \mathcal{K}^M_{n,X}\rightarrow 0\] induced by the reduction $\mathcal{O}_{X_A}\rightarrow \rho_*\mathcal{O}_X$ and with $\mathcal{K}^M_{n,\rho}$ defined to be the appropriate kernel sheaf. This exact sequence is, moreover, right-split by the structure map \[\mathcal{O}_X\rightarrow s_*\mathcal{O}_{X_A}\rightarrow s_*\rho_*\mathcal{O}_X=\mathcal{O}_X.\]

\begin{lem}\label{lem: sheafify}
For each $n\geq 2$, there is an isomorphism of sheaves of abelian groups \[\psi^n:\mathcal{K}^M_{n,\rho}\xrightarrow{\sim} \R^{n-1}_\rho/d\R^{n-2}_\rho\] coming from sheafifying the Bloch maps of \cite[Theorem 2.10]{MR3859378}. $\hfill\square$
\end{lem}

From now on we identify $\Omega_{X_A/k}^1$ with the sum \[\Omega^1_{X_A/k}\cong \pi_1^* \Omega_{X/k}^1 \oplus \pi_2^* \Omega^1_{A/k}\] via the isomorphism of \cite[\href{https://stacks.math.columbia.edu/tag/01V1}{Tag 01V1}]{stacks-project}. Similarly we identify \[\tag{K{\"u}n}\label{Kun}\Omega^n_{X_A/k}\cong \bigoplus_{j=0}^n \Omega_{X/k}^j \boxtimes \Omega_{A/k}^{n-j}\] where for any $0\leq j \leq n$ we use the notation \[\Omega_{X/k}^j \boxtimes \Omega_{A/k}^{n-j}:=\pi_1^*\Omega_{X/k}^j\otimes_{\mathcal{O}_{X_A}} \pi_2^*\Omega_{A/k}^{n-j}\quad\left(\,\,\cong \Omega_{X/k}^j\otimes_k\Omega_{A/k}^{n-j}\right)\] for the exterior product of sheaves on the product $X_A=X\times_k \mathrm{Spec}(A)$. The following two lemmas are direct generalizations from the case when $n=1$ which is observed in \cite[\S3]{MR371891}.

\begin{lem}\label{lem: splitting}
The composition \[ \R^n_\rho\rightarrow \Omega^n_{X_A/k}\twoheadrightarrow \bigoplus_{j=0}^{n-1} \Omega_{X/k}^j \boxtimes \Omega_{A/k}^{n-j}\] of the natural inclusion and projection induces a short exact sequence \[ 0\rightarrow \Omega_{X/k}^n\boxtimes \mathfrak{m}_A \rightarrow \R^n_\rho \rightarrow\bigoplus_{j=0}^{n-1} \Omega_{X/k}^j \boxtimes \Omega_{A/k}^{n-j} \rightarrow 0.\] This exact sequence is, moreover, split.$\hfill\square$
\end{lem}

From now on we write $\mathfrak{m}^0_A=\ker(\mathfrak{m}_A\xrightarrow{d} \Omega_{A/k}^1)$ for the given kernel. 

\begin{lem}\label{lem: sqses}
For any $n\geq 1$, the following square of differentials and inclusions is commutative.
\[\tag{$\ast$}\label{square}\begin{tikzcd}
\Omega^{n-1}_{X/k}\otimes_k \mathfrak{m}_A^0 \arrow{r}\arrow["d\otimes 1"]{d} & \R^{n-1}_\rho\arrow["d"]{d} \\ \Omega_{X/k}^n\boxtimes \mathfrak{m}_A\arrow{r} & \R^n_\rho
\end{tikzcd}\]

From the square \emph{(}\ref{square}\emph{)} one also gets an exact sequence \[ 0\rightarrow\frac{\Omega_{X/k}^n\boxtimes \mathfrak{m}_A}{ \mathrm{Im}(d\otimes 1)} \rightarrow \frac{\R^n_\rho}{d\R^{n-1}_\rho} \rightarrow \mathcal{C}^n_\rho\rightarrow 0\] where \[\mathcal{C}^n_\rho=\frac{\bigoplus_{j=0}^{n-1} \Omega_{X/k}^j \boxtimes \Omega_{A/k}^{n-j}}{\Omega_{X/k}^{n-1}\otimes_k d(\mathfrak{m}_A) + d\left(\bigoplus_{j=0}^{n-2}\Omega_{X/k}^j \boxtimes \Omega_{A/k}^{n-1-j}\right) }\]is the canonical quotient sheaf.
\end{lem}

\begin{proof}
To see that the square (\ref{square}) above is commutative, it suffices to observe that the differential $d$ of the de Rham complex $\Omega_{X_A/k}^\bullet$ is identifiable, via the isomorphism of (\ref{Kun}), with the differential of the total complex $\mathrm{Tot}(\Omega_{X/k}^\bullet \boxtimes \Omega_{A/k}^\bullet)$; for more information, one can consult \cite[\href{https://stacks.math.columbia.edu/tag/0FM9}{Tag 0FM9},\href{https://stacks.math.columbia.edu/tag/012Z}{Tag 012Z}]{stacks-project}. It follows from the existence of (\ref{square}) that there is a commutative diagram with exact rows
\[\begin{tikzcd}
0\arrow{r}&	\Omega^{n-1}_{X/k}\otimes_k \mathfrak{m}_A^0 \arrow{r}\arrow["d\otimes 1"]{d} & \R^{n-1}_\rho\arrow["d"]{d}\arrow{r} & \mathcal{E} \arrow{r}\arrow["\tilde{d}"]{d} & 0  \\ 0\arrow{r} &\Omega_{X/k}^n\boxtimes \mathfrak{m}_A\arrow{r} & \R^n_\rho\arrow{r} & \bigoplus_{j=0}^{n-1} \Omega_{X/k}^j \boxtimes \Omega_{A/k}^{n-j} \arrow{r}& 0
\end{tikzcd}\] where we use the placeholder \[\mathcal{E}= \Omega_{X/k}^{n-1}\otimes_k \left(\mathfrak{m}_A/\mathfrak{m}_A^0\right) \oplus \bigoplus_{j=0}^{n-2} \Omega_{X/k}^j \boxtimes \Omega_{A/k}^{n-1-j}\] and write $\tilde{d}$ for the induced map forming the rightmost vertical arrow.

We prove in Lemma \ref{lem: sumssections} below that the kernel of the composition \[\mathcal{R}^{n-1}_\rho\xrightarrow{d} \mathcal{R}_\rho^n \rightarrow \bigoplus_{j=0}^{n-1} \Omega_{X/k}^j \boxtimes \Omega_{A/k}^{n-j}\] is locally generated by both $\mathrm{ker}(d)$ and $\Omega_{X/k}^{n-1}\otimes_k \mathfrak{m}_A^0$. This implies that, for each point $x\in X_A$, any local section $\psi$ of $\mathcal{E}_x$ in the kernel of $\tilde{d}$ can be lifted to a local section $\psi'\in (\mathcal{R}_\rho^{n-1})_x$ that can be written as a sum of local sections from $\mathrm{ker}(d)_x$. The final claim of the lemma, regarding the given exact sequence, then follows from an application of the snake lemma to the above commutative diagram. 
\end{proof}

\begin{lem}\label{lem: sumssections}
With notation as in Lemma \ref{lem: sqses}, let $x\in X_A$ be a point and let $\psi\in (\mathcal{R}_{\rho}^{n-1})_x$ be a local section such that $d(\psi)$ is contained in $(\Omega_{X/k}^n\boxtimes \mathfrak{m}_A)_x\subset (\mathcal{R}^n_\rho)_x$. Then $\psi\in \mathrm{ker}(d)_x+(\Omega_{X/k}^{n-1}\otimes_k \mathfrak{m}_A^0)_x$.
\end{lem}

\begin{proof}
Fix a basis $\{e_i\}_{i\in I}$ for the $k$-vector space $\mathfrak{m}_A^0$ and extend this to a basis $\{e_i\}_{i\in J}$ for the $k$-vector space $\mathfrak{m}_A$. We note that it's possible to have $\mathfrak{m}_A^0=0$, in which case we have $I=\emptyset$, and we can assume that $\mathfrak{m}_A\neq 0$ since if $\mathfrak{m}_A=0$ then $A=k$ and the claim is immediate.

By general properties of K{\"a}hler differentials, in particular \cite[\href{https://stacks.math.columbia.edu/tag/02HP}{Tag 02HP}]{stacks-project}, we have that $\Omega_{A/k}^1$ is generated as an $A$-module by the elements $\{de_i\}_{i\in J}$.
Considering the $A$-module $\Omega_{A/k}^1$ as a $k$-vector space, it follows that the elements \[\{de_i\}_{i\in J\setminus I}, \quad \mbox{and}\quad  \{e_jde_i\}_{j\in J, i\in J\setminus I}\] span all of $\Omega_{A/k}^1$. The elements $\{de_i\}_{i\in J\setminus I}$ are $k$-linearly independent by construction so, by restricting to a subset $J'\subset J$, we can suppose that the collection \[\{de_i\}_{i\in J\setminus I}\cup \{e_jde_i\}_{j\in J', i\in J\setminus I}\] forms a basis for $\Omega_{A/k}^1$.

Using the splitting of $\mathcal{R}^{n-1}_\rho$ given in Lemma \ref{lem: splitting}, an arbitrary local section $\psi\in (\mathcal{R}^{n-1}_\rho)_x$ can be written as a sum \[\psi=\sum_{i\in I} \omega_i\otimes e_i + \sum_{j\in J\setminus I} \omega_j\otimes e_j + \sum_{r\in J\setminus I} \alpha_r \otimes de_r + \sum_{r\in J\setminus I, s\in J'} \alpha_{r,s} \otimes e_sde_r +\psi'\] where we have: \begin{itemize}[leftmargin=*, label={$\cdot$}]
\item $\omega_i\in \Omega_{X/k,x}^{n-1}$ for all $i\in I$,
\item $\omega_j \in \Omega_{X/k,x}^{n-1}$ for all $j\in J\setminus I$,
\item $\alpha_r\in \Omega_{X/k,x}^{n-2}$ for all $r\in J\setminus I$,
\item $\alpha_{r,s}\in \Omega_{X/k,x}^{n-2}$ for all pairs $(r,s)\in (J\setminus I)\times J'$,
\item and for a uniquely determined local section $\psi'$ of $(\mathcal{R}^{n-1}_\rho)_x$ contained in the summand coming from $\bigoplus_{j=2}^{n-1} \Omega_{X/k}^{n-1-j} \boxtimes \Omega_{A/k}^{j}$.
\end{itemize} Applying the differential $d$ to the local section $\psi$ then gives:
\begin{align*}
d(\psi)=& \sum_{i\in I} d\omega_i \otimes e_i\\
& +\sum_{j\in J\setminus I} d\omega_j\otimes e_j + (-1)^{n-1} \omega_j\otimes de_j\\ & + \sum_{r\in J\setminus I} d\alpha_r \otimes de_r \\&
+\sum_{r\in J\setminus I, s\in J'} d(\alpha_{r,s})\otimes e_sde_r +(-1)^{n-2} \alpha_{r,s}\otimes de_s\wedge de_r \\
&+d\psi'.
\end{align*}
Combining terms by their appropriate bidegrees, we can write:
\begin{multline*} d(\psi)- \psi''=\\ \sum_{i\in J} d\omega_i \otimes e_i + \sum_{j\in J\setminus I} (d\alpha_j +(-1)^{n-1}\omega_j) \otimes de_j
+\sum_{r\in J\setminus I, s\in J'} d(\alpha_{r,s})\otimes e_sde_r\end{multline*} for a local section $\psi''$ of $(\mathcal{R}^n_\rho)_x$ contained in the sum of those summands where $\Omega^j_{A/k}$ appears as a factor for any $j\geq 2$.

Now suppose that $d(\psi)$ is contained in $(\Omega_{X/k}^n\boxtimes \mathfrak{m}_A)_x$. It follows that $\psi''=0$, there is vanishing $d(\alpha_{r,s})=0$ for all $r\in J\setminus I$, and we get an equality $d\alpha_j+(-1)^{n-1}\omega_j=0$ for all $j\in J\setminus I$. From the latter of these we find \[\omega_j=d((-1)^n\alpha_j)\] for all $j\in J\setminus I$. From this we find that both \[ \sum_{j\in J\setminus I} \omega_j\otimes e_j + \sum_{r\in J\setminus I} \alpha_r \otimes de_r \quad \mbox{and}\quad \sum_{r\in J\setminus I, s\in J'} \alpha_{r,s} \otimes e_sde_r +\psi'\] are elements of $\mathrm{ker}(d)_x$. This immediately implies the claim, comparing with the presentation of $\psi$ above.
\end{proof}

\begin{exmp}
	If $A=k[\epsilon]/(\epsilon^2)$ is the ring of dual numbers, then both $\mathfrak{m}_A^0=0$ and $\mathcal{C}^n_\rho=0$. It follows that \[\HH^i(X,\R^n_\rho/d\R^{n-1}_\rho)=\HH^i(X,\Omega^n_{X/k})\] for any $i\geq 0$ and for any $n\geq 1$ in this case.
\end{exmp}

\begin{exmp}
	In the case $n=1$, the sheaf $\mathcal{C}^n_\rho$ simplifies to
	\[\mathcal{C}^1_\rho= \mathcal{O}_X\otimes_k \left(\Omega_{A/k}^1/d(\mathfrak{m}_A)\right).\] This is the case that's studied in \cite{MR371891}; explicitly, the sheaf $\mathcal{C}^1_\rho$ appears in \cite[(3.2)]{MR371891}.
\end{exmp}

\begin{exmp}
For $n=2$ the above sheaf $\mathcal{C}^n_\rho$ becomes \[\mathcal{C}^2_\rho= \frac{\left(\Omega_{X/k}^1 \boxtimes \Omega_{A/k}^1\right)\oplus \left(\mathcal{O}_X\boxtimes \Omega^2_{A/k}\right)}{\Omega_{X/k}^1\otimes_k d(\mathfrak{m}_A) + d(\mathcal{O}_X\boxtimes\Omega_{A/k}^1)}.\]
We study this sheaf in more detail in the next section.
\end{exmp}

\section{Weight 3}
We write $\mathscr{T}_X^{i,n}$ for the functor which assigns to an Artin local $k$-algebra $(A,\mathfrak{m}_A)$ with residue field $A/\mathfrak{m}_A\cong k$, as above, the group \[\mathscr{T}_{X}^{i,n}(A) = \HH^i(X,\mathcal{K}^M_{n,\rho}).\]
Our goal in this section is to prove the following theorem.

\begin{thm}\label{thm: Bloch3}
Suppose that $X$ is a smooth, proper, and geometrically connected $k$-variety. Fix an integer $j\geq 1$, and suppose also that the following conditions are satisfied:
\begin{enumerate}
\item $\HH^j(X,\mathcal{O}_X)=\HH^{j+1}(X,\mathcal{O}_X)=\HH^{j+2}(X,\mathcal{O}_X)=0$,
\item $\HH^j(X,\Omega_{X/k}^1)=\HH^{j+1}(X,\Omega_{X/k}^1)=0$.
\end{enumerate}
Then there is a canonical isomorphism \[\mathscr{T}_X^{j,3}(A)=\HH^j(X,\Omega^2_{X/k})\otimes_k \mathfrak{m}_A\] for any Artinian local $k$-algebra $(A,\mathfrak{m}_A)$ as above. In particular, this implies that the functor $\mathscr{T}_X^{j,3}$ is pro-representable.
\end{thm}

\begin{exmp}
	The assumptions of Theorem \ref{thm: Bloch3} are satisfied if $j=3$ and $X$ is a smooth complete intersection of two quadrics in $\mathbb{P}^7$ or if $X$ is a smooth cubic hypersurface in $\mathbb{P}^6$, see \cite[\S2]{MR309943}. If $j=3$ still, the assumptions are also satisfied when $X$ is a Gushel-Mukai fivefold \cite[Proposition 3.1]{MR4032203}. 
\end{exmp}

The proof of Theorem \ref{thm: Bloch3} relies crucially on the next two lemmas.

\begin{lem}\label{lem: vanc2}
Suppose that $X$ is geometrically integral. Fix some $j\geq 1$. Assume also that:
\begin{enumerate}
\item $\HH^j(X,\mathcal{O}_X)=\HH^{j+1}(X,\mathcal{O}_X)=0$,
\item and $\HH^j(X,\Omega_{X/k}^1)=\HH^{j+1}(X,\Omega_{X/k}^1)=0$.
\end{enumerate}
Then we have $\HH^j(X,\mathcal{C}_\rho^2)=0$, and hence also \[\HH^j(X,\mathcal{R}_\rho^2/d\mathcal{R}^1_\rho)= \HH^j(X,\Omega_{X/k}^2\boxtimes \mathfrak{m}_A/d\Omega^1_{X/k}\otimes_k \mathfrak{m}_A^0),\] for any Artinian local $k$-algebra $(A,\mathfrak{m}_A)$ with $A/\mathfrak{m}_A\cong k$ as above.
\end{lem}

\begin{proof}
To prove the lemma we'll use the exact sequence
\[\begin{tikzcd}[column sep = small]
	0\arrow[r]
	&\arrow[d, phantom," "{coordinate, name = Z}] \left(\Omega_{X/k}^1\otimes_k d(\mathfrak{m}_A)\right) \cap d(\mathcal{O}_X\boxtimes\Omega_{A/k}^1)\rightarrow
	\left(\Omega_{X/k}^1\otimes_k d(\mathfrak{m}_A)\right) \oplus d(\mathcal{O}_X\boxtimes\Omega_{A/k}^1)\arrow[d, rounded corners,
	to path={  ([xshift=17ex]\tikztostart.south)
		|- (Z) [near end]\tikztonodes
		-| ([xshift=-2ex]\tikztotarget.west)
		-- (\tikztotarget)}] \\			 
	&\left(\Omega_{X/k}^1 \boxtimes \Omega_{A/k}^1\right)\oplus \left(\mathcal{O}_X\boxtimes \Omega^2_{A/k}\right) \longrightarrow
	\mathcal{C}^2_\rho \longrightarrow
	0 & & &
\end{tikzcd}\] where the second arrow is $(+,-)$ and the third is the sum $(a,b)\mapsto a+b$. We're going to identify the cohomology groups of each of these sheaves; then we'll patch together some long-exact cohomology sequences and deduce the result.

To simplify our notation, we write \[D=\left(\Omega_{X/k}^1\otimes_k d(\mathfrak{m}_A)\right) \cap d(\mathcal{O}_X\boxtimes \Omega_{A/k}^1)\] and we set $K$ to be the kernel of the third arrow. By splicing the four-term sequence above we get two short exact sequences
\begin{equation}\tag{S1}\label{eq: seq1}
0\rightarrow D\rightarrow \left(\Omega_{X/k}^1\otimes_k d(\mathfrak{m}_A)\right)\oplus d(\mathcal{O}_X\boxtimes \Omega_{A/k}^1)\rightarrow K \rightarrow 0
\end{equation}
and
\begin{equation}\tag{S2}\label{eq: seq2}
0\rightarrow K \rightarrow \left(\Omega_{X/k}^1 \boxtimes \Omega_{A/k}^1\right)\oplus \left(\mathcal{O}_X\boxtimes \Omega^2_{A/k}\right)\rightarrow \mathcal{C}_\rho^2\rightarrow 0.
\end{equation}

The first nontrivial term $D$ can be identified as \[D=\left(\Omega_{X/k}^1\otimes_k d(\mathfrak{m}_A)\right) \cap d(\mathcal{O}_X\boxtimes\Omega_{A/k}^1)=d(\mathcal{O}_X)\otimes_k d(\mathfrak{m}_A)\subset \Omega_{X/k}^1 \boxtimes \Omega_{A/k}^1.\] By \cite[Remark 2.9.1]{MR0463157} sheaf cohomology commutes with arbitrary (especially finite) direct sums, so this yields an isomorphism \begin{align*}\HH^i(X,D) &= \HH^i(X,d(\mathcal{O}_X)\otimes_k d(\mathfrak{m}_A))\\ &\cong\HH^i(X,d(\mathcal{O}_X))\otimes_k d(\mathfrak{m}_A).\end{align*}
Since $X$ is geometrically integral, hence irreducible, the constant sheaf $\underline{k}$ is flasque. From the long exact cohomology sequence associated to the short exact sequence \[0\rightarrow \underline{k} \rightarrow \mathcal{O}_X\xrightarrow{d} d(\mathcal{O}_X)\rightarrow 0\] this, in turn, implies that, for any $i\geq 1$, there is an isomorphism \[ \HH^i(X,d(\mathcal{O}_X))\cong \HH^i(X,\mathcal{O}_X)\]
since the higher cohomology of a flasque sheaf vanishes by \cite[Proposition 2.5]{MR0463157}. Altogether, this produces an isomorphism \[\HH^i(X,D)\cong \HH^i(X,\mathcal{O}_X)\otimes_k d(\mathfrak{m}_A)\] for any $i\geq 1$.

To compute the cohomology of the middle term in the sequence (\ref{eq: seq1}), we need to write $d(\mathcal{O}_X \boxtimes \Omega_{A/k}^1)$ in a way that allows us to compute its cohomology. 
But note that the differential \[\mathcal{O}_X\boxtimes \Omega^1_{A/k}\xrightarrow{d} \left(\Omega_{X/k}^1 \boxtimes \Omega_{A/k}^1\right)\oplus \left(\mathcal{O}_X\boxtimes \Omega^2_{A/k}\right)\] has kernel $\underline{k} \otimes_k L$ where $L=\ker(\Omega^1_{A/k}\rightarrow \Omega^2_{A/k})$. By an argument similar to before, noting that $\underline{k}\otimes_k L$ is also flasque, we have isomorphisms for all $i>0$ \begin{align*}
	\HH^i(X, d(\mathcal{O}_X\boxtimes \Omega^1_{A/k})) & \cong \HH^i(X,\mathcal{O}_X\boxtimes \Omega_{A/k}^1)\\ &\cong \HH^i(X, \mathcal{O}_X)\otimes_k \Omega_{A/k}^1.
\end{align*}

The long-exact cohomology sequence associated to (\ref{eq: seq1}) now breaks up into short exact sequences, for every $i\geq 1$, \begin{align*} 0\rightarrow \HH^i(X,\mathcal{O}_X)\otimes d(\mathfrak{m}_A)\rightarrow \left(\HH^i(X,\Omega_{X/k}^1)\otimes d(\mathfrak{m}_A)\right)\oplus \left(\HH^i(X,\mathcal{O}_X)\otimes \Omega_{A/k}^1\right)\\
\rightarrow\HH^i(X,K)\rightarrow 0.
\end{align*}
Our assumptions on the vanishing of the cohomology of $\mathcal{O}_X$ and $\Omega^1_{X/k}$ imply that both $\HH^j(X,K)$ and $\HH^{j+1}(X,K)$ vanish.

The long-exact sequence associated to (\ref{eq: seq2}) now shows that \[\HH^j(X,\mathcal{C}_\rho^2)\cong \left(\HH^j(X,\Omega_{X/k}^1)\otimes \Omega_{A/k}^1\right)\oplus \left(\HH^j(X,\mathcal{O}_X)\otimes \Omega_{A/k}^2\right).\] But both summands of this latter space vanish by assumption. Lastly, the claim that \[\HH^i(X,\R^2_\rho/d\R^1_\rho)=\HH^i(X,\left(\Omega_{X/k}^2\boxtimes \mathfrak{m}_A\right)/\left(d\Omega^1_{X/k}\otimes_k \mathfrak{m}_A^0\right))\] follows as a consequence of Lemma \ref{lem: sqses}.
\end{proof}

\begin{rmk}
If $X$ is a smooth, proper, and geometrically integral $k$-variety, then the map $\HH^*(X,\mathcal{O}_X)\rightarrow \HH^*(X,\Omega_{X/k}^1)$ induced by the differential vanishes (this follows from the degeneration of the Hodge to de Rham spectral sequence proved, for example, in \cite[Corollaire 2.7]{MR894379}). In this case, the above proof can be modified (with no assumptions on the vanishing of the cohomology of either $\mathcal{O}_X$ or $\Omega_{X/k}^1$) to show that \[\HH^i(X,\mathcal{C}_\rho^2)\cong \left(\HH^i(X,\Omega_{X/k}^1)\otimes_k \Omega_{A/k}^1/d(\mathfrak{m}_A)\right) \oplus \left(\HH^i(X,\mathcal{O}_X)\otimes_k \Omega_{A/k}^2/d(\Omega_{A/k}^1)\right)\] for any $i\geq 1$. Hence, if for some fixed $j\geq 1$ one has $\HH^j(X,\Omega^1_{X/k})=0$ and $\HH^j(X,\mathcal{O}_X)=0$, then $\HH^j(X,\mathcal{C}_\rho^2)=0$.
\end{rmk}

\begin{lem}\label{lem: van4}
	Suppose that $X$ is a smooth, proper, and geometrically connected $k$-variety. Fix two integers $p\geq j\geq 1$, and suppose also that the following conditions are satisfied:
	\begin{enumerate}
		\item $\HH^{p+q}(X,\Omega_{X/k}^{j-q-1})=0$ for all $0\leq q \leq j-1$
		\item and $\HH^{p+q}(X,\Omega_{X/k}^{j-q})=0$ for all $1\leq q \leq j$
	\end{enumerate}
	Then the canonical quotient map induces an isomorphism \[\HH^p(X,\Omega_{X/k}^j)\otimes_k \mathfrak{m}_A\cong \HH^p\left(X,\frac{\Omega_{X/k}^j\boxtimes \mathfrak{m}_A}{d\Omega^{j-1}_{X/k}\otimes_k \mathfrak{m}_A^0}\right)\] for any Artinian local $k$-algebra $(A,\mathfrak{m}_A)$ with $A/\mathfrak{m}_A\cong k$ as above.
\end{lem}

\begin{proof}
	We write \[\mathscr{K}^i=\ker\left( \Omega^i_{X/k}\xrightarrow{d^i} \Omega^{i+1}_{X/k}\right)\] for the given sheaf kernel and \[\mathscr{H}^i= \HH \left(\Omega_{X/k}^{i-1}\rightarrow \Omega_{X/k}^i\rightarrow \Omega_{X/k}^{i+1}\right)\] for the given homology sheaf, i.e.\ the quotient of $\mathscr{K}^i$ by the image sheaf of the first morphism. We don't use it but, the sheaf $\mathscr{H}^i$ can also be identified with the sheafification of the presheaf associating to an open $U\subset X$ the algebraic de Rham cohomology $\HH_{dR}^i(U)$.
	
	Using this notation, there are the following exact sequences:
	\[\tag{$D_0$}\label{eq: D0} 0\rightarrow \underline{k}\rightarrow \mathcal{O}_X\rightarrow d\mathcal{O}_X\rightarrow 0,\]
	\[\tag{$H_i$}\label{eq: Hi} 0\rightarrow d\Omega_{X/k}^{i-1}\rightarrow \mathscr{K}^i\rightarrow \mathscr{H}^i\rightarrow 0\quad \mbox{for } i\geq 1 ,\]
	\[\tag{$D_i$}\label{eq: Di} 0\rightarrow \mathscr{K}^i\rightarrow \Omega_{X/k}^i\rightarrow d\Omega_{X/k}^i\rightarrow 0 \quad \mbox{for } 1\leq i\leq j-1,\]
	\[\tag{$D_j$}\label{eq: dj} 0\rightarrow d\Omega_{X/k}^{j-1}\otimes_k \mathfrak{m}_A^0\rightarrow \Omega_{X/k}^j\boxtimes \mathfrak{m}_A\rightarrow \frac{\Omega_{X/k}^j\boxtimes \mathfrak{m}_A}{d\Omega^{j-1}_{X/k}\otimes_k \mathfrak{m}_A^0} \rightarrow 0.\]
	
	The long exact sequence associated to (\ref{eq: dj}) shows that, in order to prove the lemma, it suffices to show the simultaneous vanishing \[\HH^p(X,d\Omega_{X/k}^{j-1})=\HH^{p+1}(X,d\Omega_{X/k}^{j-1})=0.\]	Looking at the sequence (\ref{eq: Di}) when $i=j-1$, this is implied by the simultaneous vanishing \[\HH^p(X,\Omega^{j-1}_{X/k})=\HH^{p+1}(X,\mathscr{K}^{j-1})=0\]\[\mbox{and } \HH^{p+1}(X,\Omega_{X/k}^{j-1})=\HH^{p+2}(X,\mathscr{K}^{j-1})=0.\]
	
	The cohomology of $\Omega_{X/k}^{j-1}$ vanishes by assumption and, from the long exact sequence of cohomology associated with (\ref{eq: Hi}) when $i=j-1$, the vanishing of the cohomology of $\mathscr{K}^{j-1}$ is implied by the simultaneous vanishing \[\HH^{p+1}(X,d\Omega^{j-2}_{X/k})=\HH^{p+1}(X,\mathscr{H}^{j-1})=0\]
	\[\mbox{and }\HH^{p+2}(X,d\Omega^{j-2}_{X/k})=\HH^{p+2}(X,\mathscr{H}^{j-1})=0.\] According to \cite[Corollary 6.2]{MR412191}, the cohomology $\HH^a(X,\mathscr{H}^b)=0$ vanishes whenever $a>b$ (which is the case when $a=p$ and $b=j-1$ by assumption). The claim then follows from repeating this argument, eventually reducing to a computation of the terms of the long exact sequence associated to (\ref{eq: D0}).
\end{proof}

\begin{proof}[Proof of Theorem \ref{thm: Bloch3}]
The isomorphism \[\mathscr{T}_X^{j,3}(A)\cong\HH^j(X,\Omega^2_{X/k})\otimes_k \mathfrak{m}_A\] follows immediately from Lemmas \ref{lem: sheafify}, \ref{lem: sqses}, \ref{lem: vanc2}, and \ref{lem: van4}.

That $\mathscr{T}_X^{j,3}$ is pro-representable can be checked using Schlessinger's Criterion \cite[Theorem 2.11]{MR217093}. We'll show, instead, that, under the assumptions of the theorem statement, the functor $\mathscr{T}_X^{j,3}$ has a tangent-obstruction theory; the sufficiency of this condition to guarantee pro-representability is proved in \cite[Corollary 6.3.5]{MR2223408}. Specifically, we'll show that there exist finite dimensional $k$-vector spaces $T_1,T_2$ and an exact sequence \[\tag{T-O}\label{t-o} 0\rightarrow T_1\otimes_k J\rightarrow \mathscr{T}_X^{j,3}(A')\rightarrow \mathscr{T}_X^{j,3}(A)\rightarrow T_2\otimes_k J\] whenever there exists a small-extension \[0\rightarrow J\rightarrow A'\rightarrow A\rightarrow 0\] of local Artinian $k$-algebras $(A,\mathfrak{m}_A)$ and $(A',\mathfrak{m}_{A'})$ with residue fields isomorphic to $k$ (recall that $A'$ is a small extension of $A$ if $J\cdot \mathfrak{m}_{A'}=0$). It will be clear from the construction that the exact sequence (\ref{t-o}) is functorial in small-extensions, and this will complete the proof.

So, assume we're in the setting above with $A,A',J$ given. We'll see that $T_1\cong \HH^j(X,\Omega_{X/k}^2)$ and $T_2=0$. From the surjection $A'\rightarrow A$ we get a commutative ladder with exact rows:
\[\begin{tikzcd}
	0\arrow{r}&	\mathscr{T}_X^{j,3}(A') \arrow{r}\arrow{d} & \HH^i(X_{A'},\mathcal{K}^M_{3,X_{A'}}) \arrow{d}\arrow{r} & \HH^i(X,\mathcal{K}^M_{3,X})  \arrow{r}\arrow[equals]{d} & 0  \\ 0\arrow{r} &\mathscr{T}_X^{j,3}(A)\arrow{r} & \HH^i(X_A,\mathcal{K}^M_{3,X_A}) \arrow{r} & \HH^i(X,\mathcal{K}^M_{3,X}) \arrow{r}& 0
\end{tikzcd}\] In view of the isomorphisms \[\mathscr{T}_X^{j,3}(A')\cong\HH^j(X,\Omega^2_{X/k})\otimes_k \mathfrak{m}_{A'}\quad\mbox{and} \quad\mathscr{T}_X^{j,3}(A)\cong\HH^j(X,\Omega^2_{X/k})\otimes_k \mathfrak{m}_A,\] the leftmost vertical arrow is a surjection with kernel $\HH^j(X,\Omega_{X/k}^2)\otimes_k J$.
This immediately proves the theorem, by the comments above.
\end{proof}
\bibliographystyle{amsalpha}
\bibliography{bib}

\providecommand{\bysame}{\leavevmode\hbox to3em{\hrulefill}\thinspace}
\providecommand{\MR}{\relax\ifhmode\unskip\space\fi MR }
\providecommand{\MRhref}[2]{%
  \href{http://www.ams.org/mathscinet-getitem?mr=#1}{#2}
}
\providecommand{\href}[2]{#2}
\begin{thebibliography}{{Sta}21}

\bibitem[Blo75]{MR371891}
Spencer Bloch, \emph{{$K_{2}$} of {A}rtinian {$Q$}-algebras, with application
  to algebraic cycles}, Comm. Algebra \textbf{3} (1975), 405--428. \MR{371891}

\bibitem[BO74]{MR412191}
Spencer Bloch and Arthur Ogus, \emph{Gersten's conjecture and the homology of
  schemes}, Ann. Sci. \'{E}cole Norm. Sup. (4) \textbf{7} (1974), 181--201
  (1975). \MR{412191}

\bibitem[DI87]{MR894379}
Pierre Deligne and Luc Illusie, \emph{Rel\`evements modulo {$p^2$} et
  d\'{e}composition du complexe de de {R}ham}, Invent. Math. \textbf{89}
  (1987), no.~2, 247--270. \MR{894379}

\bibitem[DK19]{MR4032203}
Olivier Debarre and Alexander Kuznetsov, \emph{Gushel-{M}ukai varieties: linear
  spaces and periods}, Kyoto J. Math. \textbf{59} (2019), no.~4, 897--953.
  \MR{4032203}

\bibitem[Dri14]{https://doi.org/10.48550/arxiv.1402.2222}
Benjamin~F. Dribus, \emph{A goodwillie-type theorem for milnor k-theory}, 2014,
  Arxiv. https://arxiv.org/abs/1402.2222.

\bibitem[FG05]{MR2223408}
Barbara Fantechi and Lothar G\"{o}ttsche, \emph{Local properties and {H}ilbert
  schemes of points}, Fundamental algebraic geometry, Math. Surveys Monogr.,
  vol. 123, Amer. Math. Soc., Providence, RI, 2005, pp.~139--178. \MR{2223408}

\bibitem[GT18]{MR3859378}
S.~O. Gorchinski\u{\i} and D.~N. Tyurin, \emph{Relative {M}ilnor {$K$}-groups
  and differential forms of split nilpotent extensions}, Izv. Ross. Akad. Nauk
  Ser. Mat. \textbf{82} (2018), no.~5, 23--60. \MR{3859378}

\bibitem[Har77]{MR0463157}
Robin Hartshorne, \emph{Algebraic geometry}, Springer-Verlag, New
  York-Heidelberg, 1977, Graduate Texts in Mathematics, No. 52. \MR{0463157}

\bibitem[MS78]{MR472795}
Hendrik Maazen and Jan Stienstra, \emph{A presentation for {$K_{2}$} of split
  radical pairs}, J. Pure Appl. Algebra \textbf{10} (1977/78), no.~3, 271--294.
  \MR{472795}

\bibitem[Rap72]{MR309943}
Michael Rapoport, \emph{Compl\'{e}ment \`a l'article de {P}. {D}eligne ``{L}a
  conjecture de {W}eil pour les surfaces {$K3$}''}, Invent. Math. \textbf{15}
  (1972), 227--236. \MR{309943}

\bibitem[Sch68]{MR217093}
Michael Schlessinger, \emph{Functors of {A}rtin rings}, Trans. Amer. Math. Soc.
  \textbf{130} (1968), 208--222. \MR{217093}

\bibitem[{Sta}21]{stacks-project}
The {Stacks project authors}, \emph{The stacks project},
  \url{https://stacks.math.columbia.edu}, 2021.

\bibitem[Sti83]{MR718076}
Jan Stienstra, \emph{On the formal completion of the {C}how group {${\rm
  CH}^{2}(X)$} for a smooth projective surface in characteristic {$0$}},
  Nederl. Akad. Wetensch. Indag. Math. \textbf{45} (1983), no.~3, 361--382.
  \MR{718076}

\end{thebibliography}
\end{document}